\documentclass[twoside,12pt,leqno]{amsproc}
\usepackage{amssymb,latexsym,enumerate,tikz}
\usetikzlibrary{patterns}
\usepackage[pagebackref]{hyperref}
\usepackage{amsrefs}    % include wasysym for \smiley{} \frownie{}
\usepackage{etoolbox}   % so \AtEndEnvironment{****} works below
\hypersetup{citecolor=purple, linkcolor=blue, colorlinks=true}
\usepackage{bookmark}
\numberwithin{table}{section}

\theoremstyle{plain}
\newtheorem{theorem}{Theorem}[section]
\newtheorem{lemma}[theorem]{Lemma}

\newtheorem{corollary}[theorem]{Corollary}

\theoremstyle{definition} % subsequent appear in roman not italics

\newtheorem{remark}[theorem]{Remark}
\AtEndEnvironment{remark}{\null\hfill$\triangle$}%  % symbol at end of remark

\renewcommand{\ge}{\geqslant}
\renewcommand{\le}{\leqslant}

\newcommand{\A}{\mathcal{A}}

\newcommand{\B}{\mathcal{B}}

\newcommand{\GL}{\mathrm{GL}}

\newcommand{\Null}{\textup{Null}}
\newcommand{\Rk}{\mathrm{Rk}}

\newcommand{\Z}{\mathbb{Z}}

\makeatletter        % begin hack so date appears with amsproc
\def\@adminfootnotes{%
  \let\@makefnmark\relax  \let\@thefnmark\relax
  \ifx\@empty\@date\else \@footnotetext{\@setdate}\fi%%   <-- added
  \ifx\@empty\@subjclass\else \@footnotetext{\@setsubjclass}\fi
  \ifx\@empty\@keywords\else \@footnotetext{\@setkeywords}\fi
  \ifx\@empty\thankses\else \@footnotetext{%
    \def\par{\let\par\@par}\@setthanks}%
  \fi}\makeatother   % end hack so date appears

\begin{document}

\hyphenation{}

\title[On the parameters of intertwining codes]{On the parameters of intertwining codes}
\author[S.\,P. Glasby, Cheryl E. Praeger]{S.\,P. Glasby and Cheryl E. Praeger}

\address[Glasby, Praeger]{
Centre for Mathematics of Symmetry and Computation,
University of Western Australia,
35 Stirling Highway,
Crawley 6009, Australia.\newline
 Email: {\tt Stephen.Glasby@uwa.edu.au; WWW: \href{http://www.maths.uwa.edu.au/~glasby/}{http://www.maths.uwa.edu.au/$\sim$glasby/}}\newline
Email: {\tt Cheryl.Praeger@uwa.edu.au; WWW: \href{http://www.maths.uwa.edu.au/~praeger}{http://www.maths.uwa.edu.au/$\sim$praeger}}
}
%\address[Praeger]
%{Centre for Mathematics of Symmetry and Computation,
%University of Western Australia,
%35 Stirling Highway,
%Crawley 6009, Australia. \newline
%Email: {\tt Cheryl.Praeger@uwa.edu.au WWW: \href{http://www.maths.uwa.edu.au/~praeger}{http://www.maths.uwa.edu.au/$\sim$praeger}}}

%\let\thefootnote\relax\footnotetext{{\bf MSC 2000 Classification:} Primary 94B65, Secondary 60C05} % so no numbered label appears
\date{\today\hfill 2010 Mathematics subject classification:
 94B65, 60C05}

\begin{abstract}
Let $F$ be a field and let $F^{r\times s}$ denote the space of $r\times s$ matrices
over $F$. Given equinumerous subsets
$\A=\{A_i\mid i \in I\}\subseteq F^{r\times r}$ and
$\B=\{B_i\mid i\in I\}\subseteq F^{s\times s}$ we call the subspace
$C(\A,\B):=\{X\in F^{r\times s}\mid A_iX=XB_i\textup{ for $i\in I$}\}$
an \emph{intertwining code}. We show that if $C(\A,\B)\ne\{0\}$, then
for each $i\in I$, the characteristic polynomials
of $A_i$ and $B_i$ and share a nontrivial factor. We give an exact formula
for $k=\dim(C(\A,\B))$ and give upper and lower bounds. This
generalizes previous work. Finally we construct
intertwining codes with large minimum distance when the field is
not `too small'.
We give examples of codes
where $d=rs/k=1/R$ is large where the minimum distance, dimension, and rate
of the linear code $C(\A,\B)$ are denoted by $d$, $k$, and $R=k/rs$,~respectively. 
%{\color{blue} To do: Give examples of families of intertwining codes
%where~$dR$ approaches~infinity.}
\end{abstract}

\maketitle
%\centerline{\noindent 2010 Mathematics subject classification: 94B65, 60C05}

\section{Introduction}\label{S1}

Let $F$ be a field and let $F^{r\times s}$ denote the space of $r\times s$ matrices
over $F$. Given equinumerous subsets
$\A=\{A_i\mid i \in I\}\subseteq F^{r\times r}$ and
$\B=\{B_i\mid i \in I\}\subseteq F^{s\times s}$ we call the subspace
$C(\A,\B):=\{X\in F^{r\times s}\mid A_iX=XB_i\textup{ for $i\in I$}\}$
an \emph{intertwining code}. The parameters of this linear code are denoted
$[n,k,d]$ where $n=rs$, $k:=\dim(C(\A,B))$ and $d$ is the
\emph{minimum distance} of $C(\A,\B)$. Given $u,v\in F^n$ the
\emph{Hamming distance}
$d(u,v)=|\{i\mid u_i\ne v_i\}|$ is the number of different coordinate entries,
and a subspace $C\le F^n$ has minimal (Hamming) distance
\[
  d(C):=\min\{d(u,v)\mid u\ne v\}
  =\min\{d(0,w)\mid w\in V\textup{ where $w\ne 0$.}\}.
\]
If $|I|=1$ we write $C(A,B)$ instead of $C(\{A\},\{B\})$.
Centralizer codes~\cite{AAKYPS} have the form $C(A,A)$ and 
twisted centralizer codes~\cites{AGPSY,AGP} have the form $C(A,\alpha A)$ where
$A\in F^{r\times s}$ and $\alpha\in F$. Intertwining
codes $C(A,B)$ are more general still, so our dimension formula 
(Theorem~\ref{T}) has particularly wide applicability. Furthermore,
the greater abundance of intertwining codes turns out to help us
construct intertwining codes with
large minimum distance, cf. Theorem~\ref{T11}
and~\cite{AGPSY}*{Theorem~3.2}. Intertwining codes have the advantage of
a short description, and fast matrix multiplication algorithms give rise
to efficient syndrome computations which, in turn, may be used for decoding as
described in~\cite{AGPSY}*{\S3}.

%Elements of $C(\A,\B)$ can be viewed as homomorphisms between representations
%of modules, e.g. representations $g_i\mapsto A_i$ and $g_i\mapsto B_i$
%of a group algebra $F\langle g_i\mid i\in I\rangle$.
Given representations $g_i\mapsto A_i$ and $g_i\mapsto B_i$
a group algebra $F\langle g_i\mid i\in I\rangle$, elements of $C(\A,\B)$
are homomorphisms between the associated modules. 
Hence Lemma~\ref{L4}
generalizes the fact that irreducible representations with distinct
characters are inequivalent.

An exact formula for $k:=\dim(C(A,B))$ is given
in Theorems~\ref{T6} and~\ref{T} of Section~\ref{S2}. The formula for $k$
is simplified by an identity involving partitions proved
in Section~\ref{S3}. Simpler upper and
lower bounds for $k$ are given in Section~\ref{S4}.  In Theorem~\ref{T11} in
Section~\ref{S3}, we give an algorithm to construct $A,B$ for which
the minimum distance is $d(C(A,B))=\lfloor r/k\rfloor{s}$.
These examples have $dR\le1$ where $R=\frac{k}{rs}$ is the rate of $C(A,B)$.
%Finally, Theorem~{\color{red}T12} constructs a
%sequence of $[rs,k_n,d_n]$-codes $C(A_n,B_n)$ with $k_n=n^2+n+1$ for which 
%$d_nR_n$ approaches infinity, where $R_n=k_n/(rs)$.

Corollary~\ref{C12} to Theorem~\ref{T11} shows that there exist matrices
$A\in F^{r\times r}$ and $B\in F^{s\times s}$ such that the intertwining code
$C(A,B)$ has dimension $\min\{r,s\}$ and minimum distance $\max\{r,s\}$.
We wonder how much this result can be improved.

\section{A formula for \texorpdfstring{$\dim_F(C(\A,\B))$}{}}\label{S2}

Throughout this section $\A=\{A_i\mid i\in I\}\subset F^{r\times r}$ and
$\B=\{B_i\mid i\in I\}\subset F^{s\times s}$ for a field $F$.
The idea underlying this section is to use the Jordan form over the
algebraic closure $\overline{F}$ of $F$ to compute $\dim_F(C(\A,\B))$.
To implement this idea we must simultaneously conjugate each
$A_i\in\A$, and each $B_i\in\B$, into Jordan form. This is always
possible when $|I|=1$.

Let $\GL(r,F)$ denote the general linear group
of $r\times r$ invertible matrices over~$F$. 
An ordered pair $(R,S)\in\GL(r,F)\times\GL(s,F)$ acts on $F^{r\times s}$ 
via $X^{(R,S)}=R^{-1}XS$. Clearly
$(X^{(R_1,S_1)})^{(R_2,S_2)}=X^{(R_1R_2,S_1S_2)}$,
$(XS_1)^{(R,S)}=X^{(R,S)}S_1^S$, and $(R_1X)^{(R,S)}=R_1^RX^{(R,S)}$.
%and $(XY)^{(R,S)}=X^{(R,S)}Y^S=X^RY^{(R,S)}$ if $r=s$.
%for $(R_1,S_1), (R_2,S_2)\in\GL(r,F)\times\GL(s,F)$.
%for $R_1,R_2\in\GL(r,F)$ and $S_1,S_2\in\GL(s,F)$.

\begin{lemma}\label{L1}
If $(R,S)\in\GL(r,F)\times\GL(s,F)$, then
\[
  C(\A,\B)^{(R,S)}=R^{-1}C(\A,\B)S=C(\A^R,\B^S)
\]
where $\A^R:=\{R^{-1}A_iR\mid i\in I\}$ and $\B^S:=\{S^{-1}B_iS\mid i\in I\}$.
\end{lemma}

\begin{proof}
%Clearly $A_iX=XB_i$ holds if and only if $A_i^RX^{(R,S)}=R^{-1}A_iXS$ is equal to
%$R^{-1}XB_iS = X^{(R,S)}B_i^S$.
For each $i\in I$, the equation $A_iX=XB_i$ is equivalent to
\[A_i^RX^{(R,S)}=(A_iX)^{(R,S)}=(XB_i)^{(R,S)}=X^{(R,S)}B_i^S.\qedhere\]
\end{proof}

Let $c_A(t)=\det(tI-A)$ be the characteristic polynomial of $A$.

\begin{lemma}\label{L4}
If $C(\A,\B)\ne\{0\}$, then $\gcd(c_{A_i}(t),c_{B_i}(t))\ne1$ for all $i\in I$.
\end{lemma}

\begin{proof}
Suppose that for some $i\in I$ we have $\gcd(c_{A_i}(t),c_{B_i}(t))=1$.
Then there exist polynomials $f(t),g(t)$ such that
$f(t)c_{A_i}(t)+g(t)c_{B_i}(t)=1$. Evaluating this equation at $t=B_i$, and noting that
$c_{B_i}(B_i)=0$, shows $f(B_i)c_{A_i}(B_i)=I$. Hence $c_{A_i}(B_i)$
is invertible. For $X\in C(\A,\B)$, we have $A_iX=XB_i$.
Thus $(\sum_{k\ge0}\alpha_kA_i^k)X=X(\sum_{k\ge0}\alpha_kB_i^k)$, for
all $\alpha_k\in F$, and therefore
$c_{A_i}(A_i)X=Xc_{A_i}(B_i)$. Since $c_{A_i}(A_i)=0$, post-multiplying by
$c_{A_i}(B_i)^{-1}$ shows that $X=0$, and hence $C(\A,\B)=\{0\}$.
\end{proof}

Henceforth when we wish to emphasize the field~$F$, we write $C_F(\A,\B)$.
Lemma 3.1 of~\cite{AGP}, in essence, says
$C_{\overline{F}}(\A,\B)=C_F(\A,\B)\otimes \overline{F}$.
This immediately yields Lemma~\ref{L2}.

\begin{lemma}\label{L2}
If $K$ is an extension field of $F$, then $\dim_F(C_F(\A,\B))=\dim_K(C_K(\A,\B))$.
In particular, $\dim_F(C_F(\A,\B))=\dim_{\overline{F}}(C_{\overline{F}}(\A,\B))$ where
$\overline{F}$ is the algebraic closure of $F$.
\end{lemma}

Lemma \ref{L2}  allows us to assume that $F$ is algebraically closed,
which we shall do for the rest of this section.
Given $A\in F^{r\times r}$ and $B\in F^{s\times s}$ define $A\oplus B$ to be
the block diagonal matrix
$\left(\begin{smallmatrix}A&0\\0&B\end{smallmatrix}\right)$,
and define $\A\oplus\B$ to be
$\{A_i\oplus B_i\mid i\in I\}\subseteq F^{(r+s)\times(r+s)}$.

\begin{lemma}\label{L3}
If $\A_1\subseteq F^{r_1\times r_1},\dots,\A_m\subseteq F^{r_m\times r_m}$ and
$\B_1\subseteq F^{s_1\times s_1},\dots,\B_n\subseteq F^{s_n\times s_n}$
are subsets, all with the same finite cardinality, then
\[
  C\left(\bigoplus_{i=1}^m\A_i,\bigoplus_{j=1}^n\B_j\right)\cong
  \bigoplus_{i=1}^m\bigoplus_{j=1}^nC(\A_i,\B_j).
\]
%If $\A,\A_1,\A_2\subseteq F^{r\times r}$ and $\B,\B_1,\B_2\subseteq F^{s\times  }$
%re equinumerous, then
%[
% C(\A_1\oplus \A_2,\B)\cong C(\A_1,\B)\oplus C(\A_2,\B)
% \quad\textup{and}\quad
% C(\A,\B_1\oplus\B_2) \cong C(\A,\B_1)\oplus C(\A,\B_2).
%]
\end{lemma}

\begin{proof}
Write $X=(X_{ij})$ as a block matrix where $X_{ij}$ has size $r_i\times s_j$.
The condition $X\in C\left(\bigoplus_{i=1}^m\A_i,\bigoplus_{j=1}^n\B_j\right)$
is equivalent to $X_{ij}\in C(\A_i,\B_j)$ for each $i,j$.
\end{proof}

\begin{corollary}\label{C4.5}
Suppose that $\A_1,\dots,\A_m$ and $\B_1,\dots,\B_n$ are as in Lemma~$\ref{L3}$,
and suppose that
%\[
%  \A_1\subseteq F^{r_1\times r_1},\dots,\A_m\subseteq F^{r_m\times r_m}
%  \quad\textup{and}\quad
%  \B_1\subseteq F^{s_1\times s_1},\dots,\B_n\subseteq F^{s_n\times s_n}
%\]
%are equinumerous subsets, and if 
for $i\ne j$, the characteristic polynomials
of matrices in $\A_i$ are coprime to the characteristic polynomials of matrices
in $\B_j$. Then
\[
  C\left(\bigoplus_{i=1}^m\A_i,\bigoplus_{j=1}^n\B_j\right)\cong
  \bigoplus_{i=1}^{\min\{m,n\}} C(\A_i,\B_i).
\]
\end{corollary}

\begin{proof}
Use Lemma~\ref{L3}, and note that $C(\A_i,\B_j)=\{0\}$ for
$i\ne j$ by Lemma~\ref{L4}.
\end{proof}

\begin{remark}\label{R6}
  Let $F$ be a finite field.
  Standard arguments, for example~\cite{S}*{p.\,168}, can be used to
  relate $\dim_{\overline{F}}(C_{\overline{F}}(\A,\B))$ to data computed over~$F$.
  This remark and Remark~\ref{R10} explain the details.
  Let $p_1(t),p_2(t),\dots$ enumerate the (monic) irreducible polynomials
  over $F$ and write $c_A(t)=\prod_{i\ge1}p_i(t)^{k_i}$ and
  $c_B(t)=\prod_{i\ge1}p_i(t)^{\ell_i}$, respectively. This gives rise to
  the $A$-invariant primary decomposition
  $F^r=\bigoplus_{i\ge1} \ker(p_i(A)^{k_i})$, and
  the $B$-invariant decomposition $F^s=\bigoplus_{i\ge1} \ker(p_i(A)^{\ell_i})$.
  Let $A_i$ be the restriction of $A$ to $\ker(p_i(A)^{k_i})$ and $B_i$
  the restriction of $B$ to $\ker(p_i(A)^{\ell_i})$. Corollary~\ref{C4.5}
  shows that  $\dim(C(A,B))=\sum_{i\ge1}\dim(C(A_i,B_i))$. The second ingredient
  involves partitions and is described in Remark~\ref{R10}.
\end{remark}

It is straightforward to see that $C(\A,\B)=\bigcap_{i\in I} C(A_i,B_i)$
where $C(A_i,B_i)$ means $C(\{A_i\},\{B_i\})$. Recall that a matrix
$A\in F^{r\times r}$ is \emph{nilpotent} if and only if $A^r=0$. We say that
$A$ is \emph{$\alpha$-potent}, where $\alpha\in F$, if $(A-\alpha I)^r=0$.
The following lemma and theorem reduce our deliberations from
$\alpha$-potent matrices to nilpotent matrices. For
$\A=\{A_i\mid i\in I\}\subseteq F^{r\times r}$, let $\A-\alpha I_r$ denote the
set $\{A_i-\alpha I_r\mid i \in I\}$.

\begin{lemma}\label{L5}
If $\A\subseteq F^{r\times r}$, $\B\subseteq F^{s\times s}$ and $\alpha\in F$, then
$C(\A,\B)=C(\A-\alpha I_r,\B-\alpha I_s)$.
\end{lemma}

\begin{proof}
For $i\in I$, $A_iX=XB_i$ holds if and only if $(A_i-\alpha I_r)X=X(B_i-\alpha I_s)$.
\end{proof}

A \emph{partition} $\lambda$ of $r$, written $\lambda\vdash r$,
is a sequence $\lambda=(\lambda_1,\lambda_2,\dots)$ of integers satisfying
\[
  \lambda_1\ge\lambda_2\ge\cdots\ge0
  \quad\textup{and}\quad \lambda_1+\lambda_2+\cdots=r.
\]
We call $\lambda_i$ the $i$th part of $\lambda$, and we usually omit parts of size zero.
 Let $N_r$ be the $r\times r$ nilpotent matrix
with all entries 0 except for an entry 1 in position $(i,i+1)$
for $1\le i<r$. Let $N_\lambda=\bigoplus N_{\lambda_i}$ where $\lambda\vdash r$.
Every nilpotent $r\times r$ matrix is conjugate to some $N_\lambda$
for a unique $\lambda\vdash r$. Furthermore, if an $r\times r$ matrix $R$
has eigenvalues $\rho_1,\dots,\rho_m$ and associated generalized eigenspaces
of dimensions $r_1,\dots,r_m$ where $r_1+\cdots+ r_m=r$, then $R$ has
Jordan form $\bigoplus_{i=1}^m (\rho_iI_{r_i}+N_{\lambda_i})$ where $\lambda_i$ is
a \emph{partition} of $r_i$ (not a part of a partition).

\begin{theorem}\label{T}
Suppose $A\in F^{r\times r}$, $B\in F^{s\times s}$ and $\gcd(c_A(t),c_B(t))$
has distinct roots $\zeta_1,\dots,\zeta_m$ in $\overline{F}$. Suppose that the
sizes of the Jordan blocks of $A$ associated with the generalized
$\zeta_i$-eigenspace of $A$
determine a partition $\alpha_i$, and the sizes of the Jordan blocks of $B$
associated with the generalized $\zeta_i$-eigenspace of $B$ determine
a partition $\beta_i$. Then
\[
  \dim(C(A,B))=\sum_{i=1}^m \dim(C(N_{\alpha_i},N_{\beta_i})).
%  =\sum_{i=1}^m\sum_{j\ge1}(\alpha'_i)_j(\beta'_i)_j.
\]
\end{theorem}

\begin{proof}
By Lemma~\ref{L2} we may assume that $F=\overline{F}$.
Let $A_i$ be the restriction of $A$ to its generalized $\zeta_i$-eigenspace
$\{v\mid v(A-\zeta_i I)^k\textup{ for some $k\ge0$}\}$. Then $A_i$
is $\zeta_i$-potent, and so determines a partition $\alpha_i$.
Similarly, the restriction $B_i$ of $B$ to the $\zeta_i$-eigenspace determines
a partition $\beta_i$. By Corollary~\ref{C4.5} and Lemma~\ref{L5}, we have
\[
  \dim(C(A,B))=\sum_{i=1}^m\dim(C(A_i,B_i))
  =\sum_{i=1}^m\dim(C(N_{\alpha_i},N_{\beta_i})).\qedhere
\]
%The result follows from Theorem~\ref{T6}.
\end{proof}

%The following theorem generalizes~\cite{AGP}*{Theorem~2.6}.

\begin{theorem}\label{T6}
Given partitions $\lambda$ of $r$ and $\mu$ of $s$, the dimension
of $C(N_\lambda,N_\mu)$ equals
\[
  \dim(C(N_\lambda,N_\mu))=\sum_{i\ge1}\sum_{j\ge1}\min\{\lambda_i,\mu_j\}.
\]
\end{theorem}

\begin{proof}
As $\lambda\vdash r$ and $\mu\vdash s$, we have $\sum_{i\ge1}\lambda_i=r$
and $\sum_{j\ge1}\mu_j=s$. Lemma~\ref{L3} shows that
$C(N_\lambda,N_\mu)\cong\bigoplus_{i\ge1}\bigoplus_{j\ge1}C(N_{\lambda_i},N_{\mu_j})$.
Taking dimensions, it suffices to show
$\dim(C(N_{\lambda_i},N_{\mu_j}))=\min\{\lambda_i,\mu_j\}$. This can be shown
by solving $N_{\lambda_i}X=XN_{\mu_j}$ for~$X$ and counting the number of free variables. Alternatively, $F^{\lambda_i}$ is a
uniserial $\langle N_{\lambda_i}\rangle$-module
with 1-dimensional quotient modules, and similarly for $F^{\lambda_j}$.
As their largest common quotient module is $F^{\min\{\lambda_i,\lambda_j\}}$, we have
$\dim(C(N_{\lambda_i},N_{\lambda_j}))=\min\{\lambda_i,\lambda_j\}$.
%in~\cite{B}*{Lemma~1} or~\cite{TA}.
\end{proof}

\begin{remark}\label{R10}
  Suppose $|F|=q$ is finite.
  Following on from Remark~\ref{R6} it suffices to consider the case where
  $c_A(t)=p(t)^{r/d}$, $c_B(t)=p(t)^{s/d}$, where $p(t)$ is irreducible over $F$
  of degree~$d$. The field $K:=F[t]/(p(t))$ has order $q^d$.
  In this case the structure of the modules $F^r=K^{r/s}$ and $F^s=K^{s/d}$
  is determined by partitions $\lambda\vdash r/d$ and $\mu\vdash s/d$.
  It turns out that $A$ is conjugate (see below) to
  $N_{\lambda,p}:=\textrm{diag}(N_{\lambda_1,p},N_{\lambda_2,p},\dots)\in F^{r\times r}$ where
  $N_{m,p}=\left(\begin{smallmatrix}C(p)&I&& \\ &\ddots&\ddots&\\&&C(p)&I\\&&&C(p)\end{smallmatrix}\right)\in F^{dm\times dm}$ and
  $C(p)\in F^{d\times d}$ is the companion matrix of~$p(t)$.
  Now $C(p)$ is conjugate in $\GL(d,K)$ to $\textrm{diag}(\zeta_1,\dots,\zeta_d)$ where $\zeta_1,\dots,\zeta_d$ are the (distinct) roots of $p(t)$ in $K$.
  It follows from Theorems~\ref{T6} and~\ref{T} that 
  \begin{equation}\label{E10}
    \dim(C(A,B))=\dim(C(N_{\lambda,p},N_{\mu,p}))
    =d\sum_{i\ge1}\sum_{j\ge1}\min\{\lambda_i,\mu_j\}.
  \end{equation}
  As an example, suppose $A$ is cyclic and $c_A(t)=p(t)^3$ where $d=\deg(p)=3$.
  In this case $r=9$ and $\lambda=(3)$. Write 
  $p(t)=t^3+p_2t^2+p_1t+p_0=(t-\zeta_1)(t-\zeta_2)(t-\zeta_3)$.
  Then $A$ is conjugate in $\GL(9,F)$ by~\cite{R} to 
  $\left(\begin{smallmatrix}C(p)&N&0\\0&C(p)&N\\0&0&C(p)\end{smallmatrix}\right)$
  where
  $C(p)=\left(\begin{smallmatrix}0&1&0\\0&0&1\\-p_0&-p_1&-p_2\end{smallmatrix}\right)$ and 
  $N=\left(\begin{smallmatrix}0&0&0\\0&0&0\\0&0&1\end{smallmatrix}\right)$.
  As $p(t)$ is separable,~\cite{R}*{Theorem~1} implies
  that $A$ is conjugate in $\GL(9,F)$ to
  $\left(\begin{smallmatrix}C(p)&I&0\\0&C(p)&I\\0&0&C(p)\end{smallmatrix}\right)$.
  Hence $A$ is conjugate in $\GL(9,K)$ to
  $\left(\begin{smallmatrix}D(\zeta_1)&0&0\\0&D(\zeta_2)&0\\0&0&D(\zeta_3)\end{smallmatrix}\right)$ where 
  $D(\zeta)=\left(\begin{smallmatrix}\zeta&1&0\\0&\zeta&1\\0&0&\zeta\end{smallmatrix}\right)$. This explains the factor of $d=\deg(p(t))$ in equation~\eqref{E10} and relates the generalized Jordan form of $A$ over $F$ to the Jordan form of $A$ over $K$.
\end{remark}

\section{Conjugate Partitions}\label{S3.5}

In this section we simplify the formula in Theorem~\ref{T6}
for $\dim(C(N_\lambda,N_\mu))$. We prove an identity in Lemma~\ref{L7.5}
involving partitions which replaces multiple sums by a single sum.
In order to state the simpler dimension formula we need to define `conjugate partitions'.
The \emph{conjugate} of $\lambda\vdash r$ is the partition
$\lambda'=(\lambda_1', \lambda_2', \dots)$ of $r$ whose parts satisfy
$\lambda'_i=|\{j\mid \lambda_j\ge i\}|$, for each $i$.
The Young diagram of $\lambda'$, is obtained from that of $\lambda$ by swapping
rows and columns as shown in Figure~\ref{F:YD}.
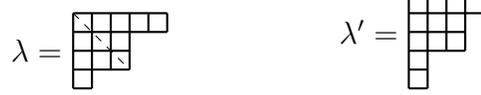
\begin{figure}[!ht]
  \caption{Young diagrams for $\lambda=(5,3,3,1)$ and
    $\lambda'=(4,3,3,1,1)$.}\label{F:YD}
  \begin{center}
  \begin{tikzpicture}[scale=0.25]
    \coordinate [label=left:\textcolor{black}{$\lambda=$}] (A) at (0,2);
    \draw[thick,black](0,0)--(0,4);
    \draw[thick,black](1,0)--(1,4);\draw[thick,black](2,1)--(2,4);
    \draw[thick,black](3,1)--(3,4);
    \draw[thick,black](4,3)--(4,4);\draw[thick,black](5,3)--(5,4);
    \draw[thick,black](0,0)--(1,0);\draw[thick,black](0,1)--(3,1);
    \draw[thick,black](0,2)--(3,2);\draw[thick,black](0,3)--(5,3);
    \draw[thick,black](0,4)--(5,4);\draw[dashed] (0,4)--(3,1);
  \end{tikzpicture}
  \hskip20mm
  \begin{tikzpicture}[scale=0.25]
    \coordinate [label=left:\textcolor{black}{$\lambda'=$}] (A) at (0,3);
    \draw[thick,black](0,0)--(0,5);\draw[thick,black](1,0)--(1,5);
    \draw[thick,black](2,2)--(2,5);\draw[thick,black](3,2)--(3,5);
    \draw[thick,black](4,4)--(4,5);
    \draw[thick,black](0,0)--(1,0);\draw[thick,black](0,1)--(1,1);
    \draw[thick,black](0,2)--(3,2);\draw[thick,black](0,3)--(3,3);
    \draw[thick,black](0,4)--(4,4);\draw[thick,black](0,5)--(4,5);
    %\draw[dashed] (0,5)--(3,2);
  \end{tikzpicture}
  \end{center}
\end{figure}

For the following result, note that
the number of nonzero $\lambda_i$ is $\lambda'_1$, and 
$r=\sum_{i=1}^{\lambda'_1}\lambda_i$.

\begin{theorem}\label{T6.5}
Given partitions $\lambda$ of $r$ and $\mu$ of $s$, the dimension
of $C(N_\lambda,N_\mu)$ equals
\[
  \dim(C(N_\lambda,N_\mu))
  =\sum_{i\ge1}\lambda'_i\mu'_i=\sum_{i=1}^{\min\{\lambda_1,\mu_1\}}\lambda'_i\mu'_i.
\]
\end{theorem}

To prove Theorem~\ref{T6.5} we need a technical lemma
which we have not been able to find in the literature, see~\cite{G}.
Lemma~\ref{L7.5} below says $\sum_{i\ge1}\lambda_i=\sum_{i\ge1}\lambda'_i$ when
$k=1$. We only need the case $k=2$ for the proof of Theorem~\ref{T6.5},
however, the argument for $k>2$ is not much harder.

\begin{lemma}\label{L7.5}
  If $\lambda, \mu,\dots,\omega$ are partitions and
$\lambda', \mu',\dots,\omega'$ are their conjugates, then
  \begin{equation}\label{E1}
  \sum_{i=1}^{\lambda'_1}\sum_{j=1}^{\mu'_1}\cdots\sum_{k=1}^{\omega'_1}\min\{\lambda_i,\mu_j,\dots,\omega_k\}=
  \sum_{i=1}^{\min\{\lambda_1,\mu_1,\dots,\omega_1\}}\lambda'_i\mu'_i\cdots\omega'_i.
  \end{equation}
\end{lemma}

\begin{proof}
By permuting the partitions $\lambda, \mu,\dots,\omega$ if necessary,
we can assume that $\lambda_1\le\mu_1\le\cdots\le\omega_1$.
If $\lambda_1=0$, then $\lambda\vdash0$ and both sides of~\eqref{E1} are zero.
If $\lambda_1=1$, then
\[
  \textup{LHS (1)}=\sum_{i=1}^{\lambda'_1}\sum_{j=1}^{\mu'_1}\cdots\sum_{k=1}^{\omega'_1}1=\lambda'_1\mu'_1\cdots\omega'_1=
  \sum_{i=1}^{\min\{\lambda_1,\mu_1,\dots,\omega_1\}}\lambda'_i\mu'_i\cdots\omega'_i
  =\textup{RHS (1).}
  \]
Suppose now that $\lambda_1>1$. We use induction on
$\lambda_1$. Let $\widehat{\lambda}$ be the partition of
$(\sum_{i\ge1}\lambda_i)-\lambda'_1$ obtained by deleting the first column of the Young
diagram of~$\lambda$. Since $1<\mu_1\le\cdots\le\omega_1$, we
define $\widehat{\mu},\dots,\widehat{\omega}$ similarly. It is clear that
$\widehat{\lambda}_i=\lambda_i-1$ for $1\le i\le \lambda'_1$ and
$\widehat{\lambda}'_i=\lambda'_{i+1}$ for $i\ge1$, and similarly for
$\widehat{\mu},\dots,\widehat{\omega}$. As $\widehat{\lambda}_1<\lambda_1$,
induction shows
\[
   \sum_{i=1}^{\widehat{\lambda}'_1}\sum_{j=1}^{\widehat{\mu}'_1}\cdots\sum_{k=1}^{\widehat{\omega}'_1}\min\{\widehat{\lambda}_i,\widehat{\mu}_j,\dots,\widehat{\omega}_k\}=
  \sum_{i=1}^{\min\{\widehat{\lambda}_1,\widehat{\mu_1},\dots,\widehat{\omega}_1\}}\widehat{\lambda}'_i\widehat{\mu}'_i\cdots\widehat{\omega}'_i.
\]
Note that $\widehat{\lambda}_i=0$ for each
$i\in[\widehat{\lambda}'_1+1,\lambda'_1]$ since
$\widehat{\lambda}'_1=\lambda'_2$, so the upper limit $\widehat{\lambda}'_1$
of the sum $\sum_{i=1}^{\widehat{\lambda}'_1}$ can be replaced by $\lambda'_1$.
Similarly, the upper limits $\widehat{\mu}'_1,\dots,\widehat{\omega}'_1$ can be
replaced by $\mu'_1,\dots,\omega'_1$. Hence, since
$\widehat{\lambda}_i=\lambda'_i-1, \dots, \widehat{\omega}_i=\omega'_i-1$,
we have
\[
   \sum_{i=1}^{\lambda'_1}\sum_{j=1}^{\mu'_1}\cdots\sum_{k=1}^{\omega'_1}\min\{\lambda_i-1,\mu_j-1,\dots,\omega_k-1\}=
  \sum_{i=1}^{\min\{\lambda_1-1,\mu_1-1,\dots,\omega_1-1\}}\lambda'_{i+1}\mu'_{i+1}\cdots\omega'_{i+1}.
\]
Re-indexing the right sum, and using $\sum_{i=1}^{\lambda'_1}\sum_{j=1}^{\mu'_1}\cdots\sum_{k=1}^{\omega'_1} (-1)=-\lambda'_1\mu'_1\cdots\omega'_1$ gives
\[
   -\lambda'_1\mu'_1\cdots\omega'_1+\sum_{i=1}^{\lambda'_1}\sum_{j=1}^{\mu'_1}\cdots\sum_{k=1}^{\omega'_1}\min\{\lambda_i,\mu_j,\dots,\omega_k\}=
  \sum_{i=2}^{\min\{\lambda_1,\mu_1,\dots,\omega_1\}}\lambda'_i\mu'_i\cdots\omega'_i.
\]
Adding $\lambda'_1\mu'_1\cdots\omega'_1$ to both sides completes the inductive
proof of~\eqref{E1}.
\end{proof}

\begin{proof}[Proof of Theorem~\ref{T6.5}]
Apply Theorem~\ref{T6} and Lemma~\ref{L7.5} with $k=2$.
\end{proof}

\section{Minimum distances}\label{S3}

In Section~\ref{S2} a formula is given for $k:=\dim(C(\A,B))$; where
we suppress mention of the field $F$ in our notation. In this section
we choose $\A$ and~$\B$ to maximize the value of the minimum
distance $d:=d(C(\A,\B))$ as a function of $k$. We focus
on the case when $|\A|=|\B|=1$. The action of $\GL(r,F)\times\GL(s,F)$
of $C(A,B)$ fixes
$k=\dim(C(A,B))$ but can change $d$ wildly, e.g. from 1 to $rs$
as setting $k=1$ in Theorem~\ref{T11} illustrates.

Let $E_{ij}$ denote the $r\times s$ matrix with all entries~0,
except the $(i,j)$ entry which is~1.

\begin{lemma}\label{L11}
Suppose $r,s,k\in\Z$ where $1\le k\le\min\{r,s\}$, and suppose $F$ is a field 
with $|F|\ge k+\min\{1,r-k\}+\min\{1,s-k\}$. 
Fix pairwise distinct scalars $\zeta_1,\dots,\zeta_k,\alpha,\beta\in F$ and set
$A_0:=\textup{diag}(\zeta_1,\dots,\zeta_k,\alpha,\dots,\alpha)\in F^{r\times r}$
and $B_0:=\textup{diag}(\zeta_1,\dots,\zeta_k,\beta,\dots,\beta)\in F^{s\times s}$.
%\[
%  A:=\textup{diag}(\zeta_1,\dots,\zeta_k,\alpha,\dots,\alpha)\in F^{r\times r}
%  \quad\textup{and}\quad
%  B:=\textup{diag}(\zeta_1,\dots,\zeta_k,\beta,\dots,\beta)\in F^{s\times s}.
%\]
Then $C(A_0,B_0)=\langle E_{11},\dots,E_{kk}\rangle$ has dimension~$k$
and minimum distance~$1$.
\end{lemma}

\begin{proof}
Note first that if $k=\min\{r,s\}$, then $A_0$ has no $\alpha$s, or $B_0$ has no
$\beta$s. Thus the assumption $|F|\ge k+\min\{1,r-k\}+\min\{1,s-k\}$ ensures
that distinct scalars $\zeta_1,\dots,\zeta_k,\alpha,\beta$ in $F$ exist.
Using a direct calculation of $C(A_0,B_0)$, or Corollary~\ref{C4.5}, shows that
$C(A_0,B_0)=\langle E_{11},\dots,E_{kk}\rangle$.
Since $d(0,E_{11})=1$, we have $d(C(A_0,B_0))=1$. 
\end{proof}

We now seek matrices $R\in\GL(r,F)$ and $S\in\GL(s,F)$ such that
$R^{-1}\langle E_{11},\dots,E_{kk}\rangle S$ has large minimum distance.
For brevity, we write $T:=R^{-1}$.

Denote the $i$th row of a matrix $A$ by $A_{i*}$ and its $j$th
column by $A_{*j}$.

\begin{lemma}\label{Lb}
Suppose $r,s,k\in\Z$ where $k\le\min\{r,s\}$. Fix $S\in F^{s\times s}$ and
$T\in F^{r\times r}$ and define $X^{(1)},\dots,X^{(k)}\in F^{r\times s}$ by
$X^{(\ell)}=T_{* \ell}S_{\ell *}$ for $1\le \ell\le k$.
%\begin{equation}\label{Eb}
%  X^{(\ell)}=T_{* \ell}S_{\ell *}  \qquad\textup{for $1\le \ell\le k$.}
%\end{equation}
Then $TE_{\ell \ell}S=X^{(\ell)}$ for $1\le \ell\le k$.
\end{lemma}

\begin{proof}
Suppose $\delta_{ij}$ is 1 if $i=j$ and 0 otherwise.
Then the $(i,j)$ entry of $E_{\ell\ell}$ is~$\delta_{i \ell}\delta_{\ell j}$.
The $(i',j')$ entry of $T_{* \ell}S_{\ell *}$ is $t_{i' \ell}s_{\ell j'}$. This
agrees with the $(i',j')$ entry of $TE_{\ell \ell}S$, namely
\[
  \sum_{i=1}^r\sum_{j=1}^s t_{i' i}\delta_{i \ell}\delta_{\ell j}s_{j j'}
  =t_{i' \ell}s_{\ell j'}.\qedhere
\]
\end{proof}

\begin{theorem}\label{T11}
Suppose $r,s,k\in\Z$ where $1\le k\le\min\{r,s\}$, and suppose $F$ is a field 
with $|F|\ge k+2$. Then there exist $A\in F^{r\times r}$ and $B\in F^{s\times s}$
such that the linear code $C(A,B)$ has dimension $k$ and minimum distance
$d=\lfloor r/k\rfloor{s}$.
\end{theorem}

\begin{proof}
By Lemma~\ref{L11} there exist diagonal matrices $A_0\in F^{r\times r}$
and $B_0\in F^{s\times s}$ such that $C(A_0,B_0)=\langle E_{11},\dots,E_{kk}\rangle$
has dimension $k$. We seek invertible matrices $R\in F^{r\times r}$ and
$S\in F^{s\times s}$ such that $A=A_0^R$ and $B=B_0^S$ give
$C(A,B)=\langle E_{11},\dots,E_{kk}\rangle^{(R,S)}$ 
with minimum distance $d=\lfloor r/k\rfloor{s}$.
Let $X^{(\ell)}=R^{-1}E_{\ell\ell}S$ for $1\le\ell\le k$.
The $r\times s$ matrices $X^{(\ell)}$,
$1\le\ell\le k$, will have a form which makes it clear that
$d=\lfloor r/k\rfloor{s}$.

First, we partition the set $\{1,\dots,r\}$ of rows into the following $k$ subsets:
\[
  I_1=\left\{1,\dots,\left\lfloor\frac{r}{k}\right\rfloor\right\},
  I_2=\left\{\left\lfloor\frac{r}{k}\right\rfloor+1,\dots,2\left\lfloor\frac{r}{k}\right\rfloor\right\},
  \dots, 
  I_k=\left\{(k-1)\left\lfloor\frac{r}{k}\right\rfloor+1,\dots,r\right\}.
\]
Choose the $i$th row of the matrix $X^{(\ell)}$ to be zero if
$i\not\in I_\ell$, and to be a vector with all $s$ entries nonzero
otherwise. Since  $\left\lfloor\frac{r}{k}\right\rfloor=|I_\ell|\le|I_k|$
for $\ell<k$, it follows that 
\[
  d(0,X^{(\ell)})=\sum_{i\in I_\ell} s
    =|I_\ell|s\ge\left\lfloor\frac rk\right\rfloor{s}
   \qquad\textup{for $1\le\ell\le k$}
\]
with equality if $\ell<k$.
The choice of these matrices is such that for each nonzero $X$ in the span
$\langle X^{(1)},\dots,X^{(k)}\rangle$ we also have
$d(0,X)\ge d(0,X^{(\ell)})$ for some $\ell$, and hence
$\langle X^{(1)},\dots,X^{(k)}\rangle$
has minimum distance $d=\lfloor r/k\rfloor{s}$.

It is well known that if the first few rows of a
square matrix are linearly independent, then the remaining rows can be chosen
so that the matrix is invertible. A similar remark holds if the
first few columns are linearly independent.
Our construction uses $k$ linearly independent $1\times s$ row vectors
$u_1,\dots, u_k$ which give the first $k$ rows of $S\in\GL(s,F)$,
and $k$ linearly independent $r\times 1$ column vectors
$v^{(1)},\dots, v^{(k)}$ which give the first $k$ columns of $R^{-1}\in\GL(r,F)$.
The pair $(R,S)$ will be used to construct $A$ and~$B$.

Henceforth suppose that $1\le\ell\le k$. 
Since $|F|\ge3$, we may choose $\gamma\in F\setminus\{1,1-s\}$.
Let $J$ be the $s\times s$ matrix with all entries~$1$. Then
the $s\times s$ matrix $S'=(\gamma -1)I+J$ is invertible as
$\det(S')=(\gamma-1)^{s-1}(\gamma+s-1)$ is nonzero. 
Let $u_\ell=(1,\dots,1,\gamma,1,\dots,1)$ be the $\ell$th row of $S'$.
Since $u_1,\dots,u_k$ are linearly independent, there exists an invertible
matrix $S\in\GL(s,F)$ with $S_{\ell*}=u_\ell$.
Of course $S=S'$ is one possibility. Similarly, let $v^{(\ell)}$ be the
$r\times1$ column vector
\[
  v^{(\ell)}_i=\begin{cases}
             1&\textup{if $i\in I_\ell$,}\\
             0&\textup{if $i\not\in I_\ell$.}\end{cases}
\]
As $v^{(1)},\dots,v^{(k)}$ are linearly independent, there exists
an $r\times r$ invertible matrix, which we call $R^{-1}$, whose first $k$
columns are $v^{(1)},\dots,v^{(k)}$. Lemma~\ref{Lb} shows that
$R^{-1}E_{\ell\ell}S=X^{(\ell)}$ for $1\le \ell\le k$.
Hence $C(A_0^R,B_0^S)=C(A_0,B_0)^{(R,S)}=\langle X^{(1)},\dots,X^{(k)}\rangle$
has minimum distance $\lfloor r/k\rfloor\,s$ as desired.
\end{proof}

\begin{corollary}\label{C12}
  If $|F|\ge \min\{r,s\}+2$, then there exist matrices
  $A\in F^{r\times r}$ and $B\in F^{s\times s}$ such that $C(A,B)$ has dimension
  $\min\{r,s\}$ and minimum distance $\max\{r,s\}$.
\end{corollary}

\begin{proof}
  Since $AX=XB$ if and only if $X^TA^T=B^TX^T$ we see that
  $C(B^T,A^T)$ equals $C(A,B)^T$. Because $C(A,B)$ and $C(A,B)^T$ have the same
  dimension and minimum distance, we may assume that $r\le s$.
  If $|F|\ge r+2$, then
  applying Theorem~\ref{T11} with $k=r$ gives the desired result.
\end{proof}

\begin{remark}
  Suitable matrices $A$ and $B$ in Theorem~\ref{T11} are constructed by first
  choosing the diagonal matrices $A_0$ and $B_0$ in Lemma~\ref{L11}, and then
  taking $A=R^{-1}A_0R$ and $B=S^{-1}B_0S$ where $R$ and $S$ are
  constructed in the proof of Theorem~\ref{T11}.
\end{remark}

It is desirable for a code to have both a high rate,
\emph{viz.} $R=k/n$, and a high distance~$d$. Can the product
$Rd$ be a constant for intertwining codes? By setting $r=s=k$ in
Theorem~\ref{T11} we obtain a rate of $R=1/r$ and a distance of $d=r$, so
the answer is affirmative. It is natural to ask how the maximum
value of~$Rd$ for an intertwining code depends on $(r,s,F)$?
We wonder whether there is a sequence $C_1,C_2,\dots$ of intertwining codes
over a field $F$ with parameters $[r_is_i,k_i,d_i]$ for
which $R_id_i=\frac{k_id_i}{r_is_i}$ approaches infinity.

\section{Upper and lower bounds for \texorpdfstring{$\dim_F(C(\A,\B))$}{}}\label{S4}

Denote that rank and nullity of $A\in F^{r\times r}$ by $\Rk(A)$ and 
$\Null(A)$, respectively. Note that $\Rk(A)+\Null(A)=r$
and $\Null(N_\lambda)=\lambda'_1$. 
In this section we bound $k=\dim(C(A,B))$ in terms of the rank and nullity
of $A$ and~$B$. If $\lambda\vdash r$ and $\mu\vdash s$, Theorem~\ref{T6} implies that
\begin{equation}\label{E9}
  \lambda'_1\mu'_1\le\sum_{i\ge1}\lambda'_i\mu'_i=\dim(C(N_\lambda,N_\mu))
  \le\left(\sum_{i\ge1}\lambda'_i\right)\left(\sum_{j\ge1}\mu'_j\right)=rs.
\end{equation}

View $A\in F^{r\times r}$ as acting on an $r$-dimensional
vector space over the algebraic closure~$\overline{F}$.
Let the $\alpha$-eigenspace, and the generalized $\alpha$-eigenspace,
of $A$  have dimensions $k_{A,\alpha}$ and $m_{A,\alpha}$, respectively.
Then $c_A(t)=\prod (t-\alpha)^{m_{A,\alpha}}$ where $m_{A,\alpha}\ne0$ for
finitely many $\alpha\in\overline{F}$ and $0\le k_{A,\alpha}\le m_{A,\alpha}$.
The following result generalizes~\cite{AGP}*{Theorems~2.8 and 4.7}.

\begin{theorem}\label{T13} 
If $A\in F^{r\times r}$ and $B\in F^{s\times s}$, then
\begin{enumerate}[{\rm (a)}]
  \item $\sum k_{A,\alpha}k_{B,\alpha}\le\dim(C(A,B)) \le\sum m_{A,\alpha}m_{B,\alpha},$
  and
  \item $(r-\Rk(A))(s-\Rk(B))\le\dim(C(A,B))\le
  (r-\Rk(A))(s-\Rk(B))+\Rk(A)\Rk(B).$
\end{enumerate}
\end{theorem}

\begin{proof}
Part~(a) follows immediately from Theorem~\ref{T} and~\eqref{E9}.

(b) The lower bound follows from part~(a) since
$r-\Rk(A)=\Null(A)=k_{A,0}$. For the upper bound, note
that $A$ is similar to a diagonal direct sum $N_\lambda\oplus A'$
where $N_\lambda$ is nilpotent of size $m_{0,A}$ and $A'$ is invertible
of size $r-m_{0,A}$.
Similarly, $B$ is similar to $N_\mu\oplus B'$ where $N_\mu$ is
nilpotent of size $m_{0,B}$ and $B'$ is invertible of size $s-m_{0,B}$.
It follows from Theorem~\ref{T} that
$\dim(C(A,B))=\dim(C(N_\lambda,N_\mu))+\dim(C(A',B'))$. Further by
Theorem~\ref{T6} $\dim(C(N_\lambda,N_\mu))=\sum_{i\ge1}\lambda'_i\mu'_i$ where,
as usual, $\lambda'$ and $\mu'$ denote conjugate partitions.
We use the observation:
\begin{equation}\label{E12}
  \textup{if $0\le x\le a$ and $0\le y\le b$, then $(a-x)(b-y)+xy\le ab$}
\end{equation}
to show that
\begin{align*}
  \dim(C(A,B))&=\lambda'_1\mu'_1+\sum_{i\ge2}\lambda'_i\mu'_i+\dim(C(A',B'))\\
              &\le \lambda'_1\mu'_1+(m_{0,A}-\lambda'_1)(m_{0,B}-\mu'_1)+(r-m_{0,A})(s-m_{0,B})\\
              &\le \lambda'_1\mu'_1+(r-\lambda'_1)(s-\mu'_1).
\end{align*}
The result follows since $\lambda'_1=\Null(N_\lambda)=\Null(A)=r-\Rk(A)$ and
$\mu'_1=s-\Rk(B)$.
\end{proof}

The Singleton bound $d+k\le n+1$ implies that if $d$ is close to $n=rs$, then
$k$ is small, and the lower bound of Theorem~\ref{T13}(b) implies that $A$ or
$B$ has high rank. Setting $k=1$ in Theorem~\ref{T11}, shows that this
bound is attained for intertwining codes.

The code $C(A,B)$ is the row nullspace of $A^T\otimes I_s+I_r\otimes B$
and the column nullspace of $A\otimes I_s+I_r\otimes B^T$ where ${}^T$ denotes
transpose.

\section*{Acknowledgments}

The authors acknowledge the contribution of Robin Chapman who emailed
us in August 2016 a proof of Theorem~\ref{T}. His formula for
$\dim(C(N_\lambda,N_\mu))$ involved a double sum which can be reduced to a
single sum using Theorem~\ref{T6.5}. The authors gratefully acknowledge
the support of the Australian Research Council Discovery Grant DP160102323.

\end{document}